\documentclass[10pt,reqno]{amsart}
     \makeatletter
     \def\section{\@startsection{section}{1}%
     \z@{.7\linespacing\@plus\linespacing}{.5\linespacing}%
     {\bfseries
     \centering
     }}
     \def\@secnumfont{\bfseries}
     \makeatother
   \usepackage{amssymb,latexsym}
\newtheorem{theorem}{Theorem}[section]

\newtheorem{proposition}[theorem]{Proposition}
\newtheorem{corollary}[theorem]{Corollary}
\theoremstyle{definition}

\theoremstyle{remark}
\newtheorem{remark}[theorem]{Remark}
\numberwithin{equation}{section} 
\makeatletter
\newcommand\@received{Received 2025-6-23;
  Accepted 2025-7-10; Communicated by the editors.}
\renewcommand\@adminfootnotes
{\let \@makefnmark \relax \let \@thefnmark \relax
\@footnotetext{\@received }
\ifx \@empty \@date \else \@footnotetext {\@setdate }\fi
\ifx \@empty \@subjclass \else \@footnotetext {\@setsubjclass }\fi
\ifx \@empty \@keywords \else \@footnotetext {\@setkeywords }\fi
\ifx \@empty \thankses \else \@footnotetext
{\def \par {\let \par \@par }\@setthanks }\fi}
\makeatother

\begin{document}

\title[The Cube of a  Gaussian Random variable]
{The Characteristic Function of the Cube of\break
 a Gaussian Random Variable}
\author[Andreas Boukas]{Andreas Boukas*}
\thanks{* Corresponding author}
\address{Andreas Boukas: Centro Vito Volterra, Universit\`{a} di Roma Tor Vergata, via Columbia  2, 00133 Roma,
Italy} \email{andreasboukas@yahoo.com}




\begin{abstract}
Using the spectral resolution of the multiplication operator  on the Schwartz class of $L^2(\mathbb{R},\mathbb{C})$, we compute the characteristic function of the cube  of a Gaussian random variable.
\end{abstract}

\maketitle

\section{Introduction}\label{intro}

The  cube of a Gaussian random variable has been of some interest due to the paper by C. Berg (see \cite{berg}) in which  it was shown that, if $X \sim \mathcal{N}(0,1/2)$ is a random variable, then $X^3$ (and, in general, $X^{2n+1}$, $n=1, 2,..$), is \textit{indeterminate}, i.e., its distribution cannot be uniquely determined by its moments.

In this paper, we consider a \textit{quantum version} of $X$, i.e. a realization of an $\mathcal{N}(0,1/2)$-distributed random variable as a self-adjoint operator on a Hilbert space and we use the \textit{spectral theorem} to compute the \textit{characteristic function} of $X^3$.  We extend to $\mathcal{N}(0,1)$-distributed random variables. No  formula for  the characteristic function of the cube  of a  Gaussian random variable seems to exist in the literature.

We consider the Hilbert space $L^2(\mathbb{R},\mathbb{C})$ with inner
product $\langle \cdot, \cdot \rangle$ and norm $\|\cdot\|$,
\begin{equation*}
\langle f, g \rangle
=\int_{\mathbb{R}}\,\overline{f(x)}g(x)\,dx
\ \,,\,\,\|f\|=\left(
\int_{\mathbb{R}}\,|f(x)|^2\,dx
\right)^{1/2}
\end{equation*}

\noindent respectively, where for a complex number $z$,   $\overline{z}$ denotes its
conjugate. The unbounded self-adjoint  \textit{multiplication  operator} $X$ on $L^2(\mathbb{R},\mathbb{C})$ is defined by
\begin{equation*}\label{qmo}
X\,f(x)=x\,f(x)\  .
\end{equation*}

\noindent The, dense in $L^2(\mathbb{R},\mathbb{C})$,  domain of $X$  is
\begin{equation*}
{\rm dom }(X)=\bigg\{f\in L^2(\mathbb{R},\mathbb{C}))
\,:\,\int_{\mathbb{R}}x^2\,|f(x)|^2\,dx<\infty\bigg\} \  .
\end{equation*}

\noindent The \textit{Schwartz class} $\mathcal{S}$ in $L^2(\mathbb{R},\mathbb{C})$ is a \textit{common invariant
domain} of  $X^n$, $n=1, 2,...$. The \textit{Gaussian function}
\begin{equation*}\label{phi}
\Phi(x)=\pi^{-1/4}\,e^{-\frac{x^2}{2}}
\end{equation*}

\noindent is a unit vector in $\mathcal{S}$. The
characteristic function of a self-adjoint operator $A$ on $L^2(\mathbb{R},\mathbb{C})$ (in the \textit{state} $\Phi$)  is
\begin{equation}\label{dcf}
 \langle  e^{itA} \rangle=  \langle \Phi, e^{itA}\Phi \rangle\,\,,\,\,t\in\mathbb{R} \  .
\end{equation}

\noindent Through \textit{Bochner's theorem} (see \cite{goldberg} and also \cite{BFspec2}-\cite{BFspec}), (\ref{dcf})  identifies $A$ with a classical random variable, and this identification is central to the connection between \textit{quantum} and \textit{classical} probability theory. Such an operator $A$ is usually referred to as a \textit{quantum random variable}, or, in standard quantum mechanical terminology, a \textit{quantum observable}.

A Lie-algebraic approach to the cube of a Gaussian random variable $X$ was taken in  \cite{AccBouCOSA}, where $X^3$ was represented in terms of a \textit{Boson pair} $a, a^{\dagger}$, as
\begin{equation*}
X^3=2^{-3/2}(a+a^{\dagger})^3
\end{equation*}

\noindent and the  characteristic function of $X^3$ was described in terms of the solution of an infinite system of  non-linear  ODE's.

\section{The Characteristic Function of $X^3$}

In \cite{BFspec2}-\cite{BFspec} it was shown, in various ways, that
\begin{equation*}
\langle  e^{i t X}  \rangle=e^{-\frac{t^2}{4}}
\end{equation*}

\noindent which means that  $X \sim \mathcal{N}(0,1/2)$. Replacing $X$ by
\begin{equation}\label{t}
T=\sqrt{2} X
\end{equation}

\noindent we find
\begin{equation*}
\langle  e^{i t T}  \rangle=\langle  e^{i (t\sqrt{2}) X}  \rangle =e^{-\frac{t^2}{2}}
\end{equation*}

\noindent so  $T \sim \mathcal{N}(0,1)$. Similarly, the operator $S$ defined by
\begin{equation}\label{tt}
S=\sigma\, T\,\,,\,\,\sigma>0\ ,
\end{equation}

\noindent satisfies
\begin{equation*}
\langle  e^{i t S}  \rangle=\langle  e^{i (\sigma t) T}  \rangle= e^{-\frac{\sigma^2 t^2}{2}}
\end{equation*}

\noindent which means that  $S \sim \mathcal{N}(0, \sigma^2)$. In general, the operator $W$ defined by
\begin{equation}\label{ttt}
S=\sigma\, T+\mu\,\,,\,\,\sigma>0\,, \,\mu\in\mathbb{R}\ ,
\end{equation}

\noindent satisfies
\begin{equation*}
\langle  e^{i t W}  \rangle=e^{it\mu}\langle  e^{i (\sigma t) T}  \rangle= e^{it\mu-\frac{\sigma^2 t^2}{2}}
\end{equation*}

\noindent so  $W \sim \mathcal{N}(\mu,\sigma^2)$.

\begin{theorem}\label{X} For $t\in \mathbb{R}$, the characteristic function of the cube of the $\mathcal{N}(0,1/2)$-distributed random variable  $X$ of (\ref{qmo})  is
\begin{equation*}\label{nd0}
\langle  e^{i t X^3}  \rangle=\left\{
\begin{array}{llr}
 \frac{2 }{3 |t| \sqrt{3 \pi}}\,e^{\frac{2}{27 t^2}} \,K_{1/3}\left(\frac{2}{27 t^2} \right)   &, \mbox{ if } t\neq 0  \\
1&, \mbox{ if } t=0
\end{array}
\right.
\end{equation*}
\noindent where, for real $\nu$ and (generally) complex $z$,  $K_{\nu}(z)$ is the modified Bessel function of the  2nd kind.
\end{theorem}

\begin{proof} The spectral resolution of $X$ (see \cite{yosida}, Sections XI.5 and XI.6) is
\begin{equation}\label{sthx}
X=\int_{\mathbb{R}}\lambda \,dE_\lambda
\end{equation}

\noindent where
\begin{equation*}
E_\lambda f(x)=\left\{
\begin{array}{llr}
 f(x)   &, \mbox{ if } x\leq \lambda   \\
0&, \mbox{ if } x>\lambda
\end{array}
\right. \  .
\end{equation*}

\noindent Thus
\begin{equation*}
X^3=\int_{\mathbb{R}}\lambda^3 \,dE_\lambda
\end{equation*}

\noindent which, letting $\xi=\lambda^3$ and $Y=X^3$,  becomes
\begin{equation*}
Y=\int_{\mathbb{R}}\xi \,dF_{\xi}
\end{equation*}

\noindent where
\begin{equation*}
F_{\xi}=E_{{\xi^{1/3}}} \ .
\end{equation*}

\noindent Thus,
\begin{equation}\label{cq}
\langle \Phi, e^{it Y} \Phi \rangle= \int_{\mathbb{R}}e^{i t \xi} \,d\langle \Phi, F_{\xi} \Phi \rangle \ .
\end{equation}

\noindent Since
\begin{equation*}
  F_{\xi} \Phi(x) = \left\{
\begin{array}{llr}
 \Phi(x)   &, \mbox{ if } x\leq \xi^{1/3}   \\
0&, \mbox{ if } x>\xi^{1/3}
\end{array}
\right.
\end{equation*}

\noindent we find

\begin{align*}
   \langle \Phi, F_{\xi} \Phi \rangle&=\int_{\mathbb{R}}\,\overline{\Phi(x)}F_{\xi} \Phi (x)\,dx=\int_{-\infty}^{\xi^{1/3}}\,\overline{\Phi(x)} \Phi (x)\,dx\\
   &=\int_{-\infty}^{\xi^{1/3}}\,| \Phi (x)|^2\,dx =\int_{-\infty}^{\xi^{1/3}}\,\pi^{-1/2}\,e^{-x^2}\,dx \notag
   \end{align*}

\noindent so, for $\xi \neq 0$,

\begin{align}\label{dxt}
      d\langle \Phi, F_{\xi} \Phi \rangle  &=\pi^{-1/2}\,e^{-(\xi^{1/3})^2}\,d(\xi^{1/3})=\pi^{-1/2}\,e^{-(\xi^{1/3})^2}\,\frac{1}{3}\,\xi^{-2/3}\,d\xi\\
  &=\frac{1}{3 \sqrt{\pi}} \xi^{-2/3} \, e^{- \xi^{2/3}}\,d\xi=\frac{1}{3 \sqrt{\pi}} |\xi|^{-2/3} \, e^{- |\xi|^{2/3}}\,d\xi \notag
\end{align}

\noindent which is the formula given in \cite{berg} for the density of the cube of a classical $\mathcal{N}(0,1/2)$-distributed random variable $X$. Thus, by (\ref{cq}),

\begin{align}
 \langle \Phi, e^{itY} \Phi \rangle&= \int_{\mathbb{R}}e^{i t \xi} \,\frac{1}{3 \sqrt{\pi}} \xi^{-2/3} \, e^{- \xi^{2/3}}\,d\xi \\
 &=\frac{1}{3 \sqrt{\pi}} \int_{\mathbb{R}}  \xi^{-2/3} \, e^{i t \xi- \xi^{2/3}}\,d\xi \notag
 \end{align}

 \noindent which, letting $x=\xi^{1/3}$, gives
 \begin{equation}\label{chf}
  \langle \Phi, e^{itY} \Phi \rangle= \frac{1}{ \sqrt{\pi}} \int_{\mathbb{R}}e^{i t x^3-x^2} \,dx \ .
\end{equation}

\noindent Since
 \begin{equation*}\label{chf2}
   \int_{\mathbb{R}}|e^{i t x^3-x^2}| \,dx=  \int_{\mathbb{R}}e^{-x^2} \,dx =\sqrt{\pi}
\end{equation*}

\noindent the  characteristic function of $Y$ exists for all $t\in\mathbb{R}$. Using \textit{Euler's formula}, (\ref{chf}) becomes
 \begin{equation}\label{chff}
  \langle \Phi, e^{itY} \Phi \rangle= \frac{1}{ \sqrt{\pi}} \int_{\mathbb{R}}\cos (tx^3)e^{-x^2} \,dx + i \frac{1}{ \sqrt{\pi}} \int_{\mathbb{R}}\sin (tx^3)e^{-x^2} \,dx \, \ .
\end{equation}

\noindent Since its integrand is an odd function of $x$, the second integral on the right hand side of (\ref{chff}) vanishes and we obtain
\begin{equation*}\label{chfff}
  \langle \Phi, e^{itY} \Phi \rangle= \frac{1}{ \sqrt{\pi}} \int_{\mathbb{R}}\cos (tx^3)e^{-x^2} \,dx
\end{equation*}

\noindent which, since the integrand is an even function of $x$, reduces to
\begin{equation*}\label{chffff}
  \langle \Phi, e^{itY} \Phi \rangle= \frac{2}{\sqrt{\pi}} \int_0^\infty\cos (tx^3)e^{-x^2} \,dx  \  .
\end{equation*}

\noindent We will show that
\begin{equation*}\label{g1}
\frac{2}{ \sqrt{\pi}} \int_0^\infty \cos (tx^3)e^{-x^2} \,dx = \left\{
\begin{array}{llr}
 \frac{2}{3 \,|t| \,\sqrt{3 \pi}}\,e^{\frac{2}{27 t^2}} \,K_{1/3}\left(\frac{2}{27 t^2} \right)   &, \mbox{ if } t\neq 0  \\
1&, \mbox{ if } t=0
\end{array}
\right.
\end{equation*}

\noindent which is the same as
\begin{equation*}\label{g11}
\frac{2}{ \sqrt{\pi}} \int_0^\infty \cos (|t|\,x^3)e^{-x^2} \,dx = \left\{
\begin{array}{llr}
 \frac{2}{3 \,|t| \,\sqrt{3 \pi}}\,e^{\frac{2}{27 |t|^2}} \,K_{1/3}\left(\frac{2}{27 |t|^2} \right)   &, \mbox{ if } t\neq 0  \\
1&, \mbox{ if } t=0
\end{array}
\right.
\end{equation*}

\noindent Thus, it suffices to consider the case $t\geq 0$ only. The case $t=0$ is a well known Gaussian integral and is clear. The case $t> 0$ is equivalent to
\begin{equation}\label{g2}
 3 \sqrt{3 }\, t\,e^{-\frac{2}{27 t^2}} \, \int_0^\infty \cos (tx^3)e^{-x^2} \,dx =
 K_{1/3}\left(\frac{2}{27 t^2} \right)      \ .
\end{equation}

\noindent To prove (\ref{g2}), let $A, B: (0. \infty) \to \mathbb{R}$ be defined by

\begin{align}
A(t)&=  3 \sqrt{3 }\, t\,e^{-\frac{2}{27 t^2}} \, \int_0^\infty \cos (tx^3)e^{-x^2} \,dx \ ,\label{dfu}\\
B(t)&= K_{1/3}\left(\frac{2}{27 t^2} \right) \notag     \ .
\end{align}

\noindent To show that $A=B$ we begin by showing that $A, B$ satisfy the same second-order differential equation. For $t > 0$ we have:
\begin{equation*}
B^\prime(t)=  \left(K_{1/3}\left(\frac{2}{27 t^2} \right)  \right)^\prime =-\frac{4}{27t^3} K_{1/3}^\prime\left(\frac{2}{27 t^2} \right)
\end{equation*}

\noindent and
\begin{equation*}
B^{\prime\prime}(t)= \frac{4}{9t^4} K_{1/3}^\prime\left(\frac{2}{27 t^2} \right)+\frac{16}{27^2t^6} K_{1/3}^{\prime\prime}\left(\frac{2}{27 t^2} \right) \  .
\end{equation*}

\noindent Since, by the definition of modified Bessel functions of the second kind,
\begin{equation*}
t^2\, K_n^{\prime\prime}(t)+t\, K_n^{\prime}(t)-(t^2+n^2)\,K_n(t)=0 \  ,
\end{equation*}

\noindent we have
\begin{equation*}
\frac{4}{27^2 t^4}\, K_{1/3}^{\prime\prime}\left(\frac{2}{27 t^2} \right) +\frac{2}{27 t^2}\, K_{1/3}^{\prime}\left(\frac{2}{27 t^2} \right) -\left(\frac{4}{27^2 t^4}+\frac{1}{9}\right)\,K_{1/3}\left(\frac{2}{27 t^2} \right) =0 \  .
\end{equation*}

\noindent Thus
\begin{equation*}
\frac{t^2}{4} B^{\prime \prime}(t)+\frac{t}{4}B^{ \prime}(t)-\left(\frac{4}{27^2 t^4}+\frac{1}{9}\right)\,B(t) =0 \  .
\end{equation*}

\noindent Similarly, letting
\begin{equation*}\label{bb}
J(t)=\int_0^\infty \cos (tx^3)e^{-x^2} \,dx
\end{equation*}

\noindent (\ref{dfu}) implies
\begin{equation*}
A^{\prime }(t)= \frac{4}{3\sqrt{3}t^3}\,e^{-\frac{2}{27 t^2}} \,J(t)+3 \sqrt{3} \,e^{-\frac{2}{27 t^2}} \,J^\prime (t)
\end{equation*}

\noindent and

\begin{align*}
 A^{ \prime\prime}(t)&= \left(\frac{16}{81\sqrt{3}t^6} - \frac{4}{\sqrt{3}t^4}  \right)\,e^{-\frac{2}{27 t^2}} \,J(t)+ \frac{8}{3\sqrt{3}t^3}\,e^{-\frac{2}{27 t^2}} \,J^\prime(t)\\
 &\,\,\,+3 \sqrt{3} \,e^{-\frac{2}{27 t^2}} \,J^{\prime \prime} (t) \notag\  .
\end{align*}

\noindent Since the integrand $g(t, x)=\cos (tx^3)e^{-x^2}$ of (\ref{bb}) and its partial derivative $g_t(t, x)=-x^3\,\sin (tx^3)e^{-x^2} $  are continuous in both $x$ and $t$, and the integral of $g_t(t, x)$ on $[0, \infty )$  is uniformly convergent for $t\in [0, \beta ]$, for all $\beta >0$ (see Theorem 25.14 of \cite{bartle}), we may differentiate $J(t)$ under the integral sign
(see Theorem 25.19 of \cite{bartle}) and we find
\begin{equation*}\label{bbb}
J^\prime(t)=-\int_0^\infty x^3\,\sin (tx^3)e^{-x^2} \,dx
\end{equation*}

\noindent and, in like manner,
\begin{equation*}\label{bbb2}
J^{\prime\prime}(t)=-\int_0^\infty x^6\,\cos (tx^3)e^{-x^2} \,dx \  .
\end{equation*}

\noindent Thus

\begin{align}\label{mbe}
 & \frac{t^2}{4} A^{\prime \prime}(t)+\frac{t}{4}A^{ \prime}(t)-\left(\frac{4}{27^2 t^4}+\frac{1}{9}\right)\,A(t) \\
 &=\frac{1}{12\sqrt{3}}e^{-\frac{2}{27 t^2}}\int_0^\infty \left((15 t-27 t^3 x^6) \cos (tx^3)-(8+81 t^2) x^3 \sin (t x^3)  \right)  e^{-x^2}\,dx \notag\\
 &=\frac{1}{12\sqrt{3}}e^{-\frac{2}{27 t^2}}\lim_{n\to \infty} \int_0^n \left((15 t-27 t^3 x^6) \cos (tx^3)-(8+81 t^2) x^3 \sin (t x^3)  \right)  e^{-x^2}\,dx \notag\\
 &=\frac{1}{12\sqrt{3}}e^{-\frac{2}{27 t^2}}\lim_{n\to \infty} \left((15 n +6 n^3)t \cos (tn^3)+(4+4 n^2-9n^4 t^2)\sin (t n^3)   \right)e^{-n^2}\notag\\
 &=\frac{1}{12\sqrt{3}}e^{-\frac{2}{27 t^2}}\cdot 0=0\notag \ .
\end{align}

 \noindent The general solution of (\ref{mbe}) is
\begin{equation}\label{ae}
A(t)=c_1 I_{1/3}\left(\frac{2}{27 t^2} \right)  +c_2 K_{1/3}\left(\frac{2}{27 t^2} \right)\,\,,\,\,c_1, c_2\in\mathbb{R} \ ,
\end{equation}

\noindent where, for real $\nu$ and (generally) complex $z$, $I_{\nu}(z)$ and $K_{\nu}(z)$ are the modified Bessel function of the 1st and 2nd kind, respectively (see \cite{arfken}, \cite{web1}).

A direct comparison of the two sides of (\ref{ae}) for the determination of $c_1, c_2$, cannot be made,   due to the fact that there is no closed form for the \textit{highly oscillatory} integral involved in the definition of $A(t)$ in (\ref{dfu}). To overcome this difficulty, we will compare asymptotic expansions of the two sides of (\ref{ae}).

 For real $z$, it is known (see e.g. formula $11.127$ of \cite{arfken}, or formula $9.7.2$ of \cite{web1}) that the modified Bessel function of the second function  $K_{\nu}(z)$ admits the asymptotic (for large $z$) expansion
\begin{equation*}
 K_{\nu}(z)\sim \sqrt{\frac{\pi}{2z}} e^{-z} \left(1+ \frac{(4\nu^2-1^2)}{1!\,8z} +
 \frac{(4\nu^2-1^2)(4\nu^2-3^2)}{2!\,(8z)^2}+\cdots    \right)
\end{equation*}

\noindent which for $\nu=1/3$ and $z=\frac{2}{27 t^2}$ becomes (for $t$ close to zero)
\begin{equation}\label{111}
 K_{1/3}\left(\frac{2}{27 t^2} \right)\sim \frac{3\sqrt{3\pi}}{2} t e^{-\frac{2}{27 t^2}} \left(1-
 \frac{15}{16}t^2+\frac{3465}{512}t^4 +\cdots    \right)  \  .
\end{equation}

\noindent The corresponding expansions for $I_{\nu}(z)$ and $I_{1/3}\left(\frac{2}{27 t^2} \right)$ are
\begin{equation*}
 I_{\nu}(z)\sim \frac{1}{\sqrt{2 \pi z}} e^{z} \left(1- \frac{(4\nu^2-1^2)}{1!\,8z} +
 \frac{(4\nu^2-1^2)(4\nu^2-3^2)}{2!\,(8z)^2}-\cdots    \right)
\end{equation*}

\noindent and
\begin{equation}\label{1111}
 I_{1/3}\left(\frac{2}{27 t^2} \right)\sim \frac{3\sqrt{3}}{2 \sqrt{\pi}} t e^{\frac{2}{27 t^2}} \left(1+
 \frac{15}{16}t^2+\frac{3465}{512}t^4 +\cdots    \right)  \  ,
\end{equation}

\noindent respectively (see  formula $9.7.1$ of \cite{web1}).

\noindent Using the Maclaurin series for the cosine function, from (\ref{dfu}) we obtain
\begin{equation*}
A(t)=  3 \sqrt{3 }\, t\,e^{-\frac{2}{27 t^2}} \, \int_0^\infty \sum_{n=0}^\infty \frac{(-1)^n}{(2n)!} t^{2n}x^{6n}e^{-x^2} \,dx
\end{equation*}

\noindent which, using
\begin{equation*}
   \int_{0}^\infty x^{6n} e^{-x^2} \,dx =\frac{1}{2}\Gamma \left(3n+\frac{1}{2}\right)=\frac{1}{2}\cdot\frac{1\cdot 3\cdot 5\cdot\cdots (6n-1)}{2^{3n}}\sqrt{\pi}
\end{equation*}

\noindent where $\Gamma$ stands for the \textit{Gamma function},
and integrating term-by-term, becomes

\begin{align}
 A(t)&=\frac{ 3 \sqrt{3 }}{2}\, t\,e^{-\frac{2}{27 t^2}} \,\sum_{n=0}^\infty \frac{(-1)^n}{(2n)!} t^{2n}\,\Gamma \left(3n+\frac{1}{2}\right)\\
  &=\frac{3 \sqrt{3 }}{2}\, t\,e^{-\frac{2}{27 t^2}} \left( \Gamma \left(\frac{1}{2}\right)- \frac{1}{2} t^{2}\,\Gamma \left(3+\frac{1}{2}\right)+\frac{1}{24} t^{4}\,\Gamma \left(6+\frac{1}{2}\right)\cdots \right)\notag\\
 &=\frac{3 \sqrt{3 }}{2}\, t\,e^{-\frac{2}{27 t^2}} \left( \sqrt{\pi}- \frac{1}{2} t^{2}\,\frac{1\cdot 3\cdot 5}{2^3}\sqrt{\pi}+\frac{1}{24} t^{4}\,\frac{1\cdot 3\cdot 5\cdot \cdots 11}{2^6}\sqrt{\pi}\cdots \right)\notag
 \end{align}

 \noindent so we have obtained the asymptotic  (for $t$ close to zero) expansion
 \begin{equation}\label{222}
 A(t)\sim \frac{3\sqrt{3\pi}}{2} t e^{-\frac{2}{27 t^2}} \left(1-
 \frac{15}{16}t^2+\frac{3465}{512}t^4 +\cdots    \right)  \  .
\end{equation}

\noindent By  (\ref{111}) and (\ref{1111}),
 \begin{equation*}
\lim_{t \to 0^+}\dfrac{I_{1/3}\left( \frac{2}{27 t^2}\right)}{K_{1/3}\left(\frac{2}{27 t^2}\right)}=\lim_{t \to 0^+}\frac{1}{\pi}e^{\frac{4}{27 t^2}} = \infty\ ,
\end{equation*}

\noindent so,  by  (\ref{ae}) and (\ref{222}),  we conclude that
\begin{equation*}
 c_1=0\,,\,c_2=1 \  .
\end{equation*}

\noindent Therefore,
\begin{equation*}
A(t)=  K_{1/3}\left(\frac{2}{27 t^2} \right) =B(t) \  .
\end{equation*}
\end{proof}

\begin{remark}\label{r1} \rm The integral form  (\ref{chf}) of the characteristic function of $Y=X^3$ could have been obtained directly from (\ref{sthx}), the spectral theorem and
\begin{equation*}
 e^{itY}=e^{it X^3}=f(X)= \int_{\mathbb{R}}f(\lambda) \,dE_\lambda=  \int_{\mathbb{R}}e^{it\lambda^3} \, dE_\lambda
\end{equation*}

\noindent as

\begin{align*}
   \langle \Phi, e^{itY} \Phi \rangle&= \int_{\mathbb{R}}e^{it\lambda^3} d\langle \Phi, \,E_\lambda \Phi \rangle=\int_{\mathbb{R}} e^{it\lambda^3} d \int_{-\infty}^\lambda |\Phi(x)|^2 \,dx \notag\\
   &=\int_{\mathbb{R}} e^{it\lambda^3}  |\Phi(\lambda)|^2 \,d\lambda =\int_{\mathbb{R}} e^{it\lambda^3} \pi^{-1/2} e^{-\lambda^2} \,d\lambda \notag\\
   &=\frac{1}{ \sqrt{\pi}} \int_{\mathbb{R}} e^{it\lambda^3-\lambda^2}  \,d\lambda \notag \  .
\end{align*}

\end{remark}

\begin{remark} \rm A power series approach to the computation of the characteristic function of $Y=X^3$ would have not worked due to the lack of uniform convergence. Specifically, we can try to  obtain a "power series form" of (\ref{chf}) as follows:

\begin{align*}
 \langle \Phi, e^{itY} \Phi \rangle&= \frac{1}{ \sqrt{\pi}} \int_{\mathbb{R}}e^{i t x^3}e^{-x^2} \,dx = \frac{1}{ \sqrt{\pi}} \int_{\mathbb{R}}\sum_{n=0}^\infty \frac{(itx^3)^n}{n!} e^{-x^2} \,dx\\
 &=\frac{1}{ \sqrt{\pi}} \sum_{n=0}^\infty \frac{(it)^n}{n!}\int_{\mathbb{R}} x^{3n} e^{-x^2} \,dx \notag\ .
 \end{align*}

\noindent Using

\begin{align*}
   \int_{\mathbb{R}} x^{3n} e^{-x^2} \,dx &=\frac{1+(-1)^{3n}}{2}\Gamma \left(\frac{3n+1}{2}\right)\\
   &=\left\{
\begin{array}{llr}
  \Gamma \left(3k+\frac{1}{2}\right)  &, \mbox{ if } n=2k,\,k=0, 1, 2,...   \\
0&, \mbox{ if }  n=2k+1,\,k=0, 1, 2,...
\end{array}
\right. \notag
\end{align*}

\noindent where $\Gamma$ stands for the \textit{Gamma function}, we find

\begin{align*}
 \langle \Phi, e^{itY} \Phi \rangle&=\frac{1}{ \sqrt{\pi}} \sum_{k=0}^\infty \frac{(it)^{2k}}{(2k)!}   \Gamma \left(3k+\frac{1}{2}\right)\\
 &=\frac{1}{ \sqrt{\pi}} \sum_{k=0}^\infty \frac{(-1)^k \, t^{2k}}{(2k)!}  \Gamma \left(3k+\frac{1}{2}\right) \notag\\
 &=\frac{1}{ \sqrt{\pi}} \sum_{k=0}^\infty \frac{(-1)^k \, t^{2k}}{(2k)!} \frac{(6k-1)!!}{2^{3k}} \sqrt{\pi} \notag  \\
&=\sum_{k=0}^\infty \frac{(-1)^k \,(6k-1)!!} {2^{3k} \,(2k)!}t^{2k}    \notag
\end{align*}

\noindent where $(6k-1)!!=1\cdot 3 \cdot 5 \cdots (6k-1)$ and $(-1)!!=1$. The radius of convergence of the above  power series is equal to zero, contradicting the fact that the improper integral converges for all $t\in \mathbb{R}$. This is because switching the integration and summation signs cannot be justified. Letting, for each  $t\in \mathbb{R}$ and $n=1, 2, ...$,
\begin{equation*}
 f_n(x)= \frac{(itx^3)^n}{n!} e^{-x^2}
\end{equation*}

\noindent we see that, for all $x\in \mathbb{R}$,
\begin{equation*}
 |f_n(x)|= \frac{|t|^n |x|^{3n}}{n!} e^{-x^2} \leq M_n
\end{equation*}

\noindent where, due to the fact that the function
\begin{equation*}
 g(x)=  |x|^{3n} e^{-x^2}
\end{equation*}

\noindent is maximized at $x=\pm \sqrt{\frac{3n}{2}}$,
\begin{equation*}
 M_n= \frac{|t|^n }{n!} \left( \frac{3n}{2}\right)^{\frac{3n}{2}} e^{-\frac{3n}{2}} \  .
\end{equation*}

 \noindent For $t\neq 0$,
\begin{equation*}
 \lim_{n\to +\infty} \left| \frac{M_{n+1}}{M_n}\right|=+\infty
\end{equation*}

\noindent so $\sum_{n=0}^\infty M_n$ diverges and  the \textit{Weierstrass $M$-test} cannot be applied to establish the uniform convergence of $\sum_{n=0}^\infty f_n(x)$.

\end{remark}

\begin{corollary}\label{x} For $t\in \mathbb{R}$, the characteristic function of the cube of the $\mathcal{N}(0,1)$-distributed  random variable  $T$ of (\ref{t}) is
\begin{equation}\label{nd1}
\langle  e^{i t T^3}  \rangle=\left\{
\begin{array}{llr}
 \frac{1 }{3 |t| \sqrt{6 \pi}}\,e^{\frac{1}{108 t^2}} \,K_{1/3}\left(\frac{1}{108 t^2} \right)   &, \mbox{ if } t\neq 0  \\
1&, \mbox{ if } t=0
\end{array}
\right.   \  .
\end{equation}
\end{corollary}

\begin{proof}
\begin{equation*}
\langle  e^{i t T^3}  \rangle=\langle  e^{i (2\sqrt{2}t) X^3}  \rangle
\end{equation*}

\noindent and (\ref{nd1}) follows from (\ref{nd0}) by replacing $t$ by $2\sqrt{2}t$.
\end{proof}

\begin{corollary}\label{xx} For $t\in \mathbb{R}$, the  characteristic function of the cube of the  $\mathcal{N}(0,\sigma^2)$-distributed  random variable  $S$ of (\ref{tt})  is
\begin{equation}\label{nd2}
\langle  e^{i t S^3}  \rangle=\left\{
\begin{array}{llr}
 \frac{1 }{3 \sigma^3|t| \sqrt{6 \pi}}\,e^{\frac{1}{108\sigma^6 t^2}} \,K_{1/3}\left(\frac{1}{108\sigma^6 t^2} \right)  &, \mbox{ if } t\neq 0  \\
1&, \mbox{ if } t=0
\end{array}
\right.   \ .
\end{equation}
\end{corollary}

\begin{proof}
\begin{equation*}
\langle  e^{i t S^3}  \rangle=\langle  e^{i (\sigma^3  t) T^3}  \rangle
\end{equation*}

\noindent and (\ref{nd2}) follows from (\ref{nd1}) by replacing $t$ by $\sigma^3 t$.
\end{proof}

\begin{remark} \rm Attempting to compute the characteristic function of the cube of the $\mathcal{N}(\mu, \sigma^2)$-distributed random variable $W$ of (\ref{ttt}) results in the spectral integral
\begin{equation*}
 e^{it W^3}= e^{it (\sigma X+\mu)^3} =f(X)= \int_{\mathbb{R}}f(\lambda) \,dE_\lambda=  \int_{\mathbb{R}}e^{it(\sigma\lambda+\mu)^3} \, dE_\lambda
\end{equation*}

\noindent and, as in Remark \ref{r1},
\begin{equation}\label{intr}
 \langle e^{it W^3} \rangle=\frac{1}{ \sqrt{\pi}}   \int_{\mathbb{R}}e^{it(\sigma\lambda+\mu)^3-\lambda^2}    \,d\lambda
\end{equation}

\noindent For $\mu \cdot \sigma \neq 0$, the integral in (\ref{intr}) seems to be analytically intractable, even with the use of symbolic computation software,  in the sense that no closed form in terms of special functions seems to exist. In such cases one can, at best, look for an asymptotic expansion.

\end{remark}

\section{Indeterminacy of $X^3$}

The indeterminacy of the distribution of the cube of a $\mathcal{N}(0,1/2)$-distributed random variable $X$, i.e. the fact that the distribution of  $X^3$  can not be uniquely determined by its moments, was established in \cite{berg} by exhibiting other probability distributions with the same moments as  $X^3$.

A suggestion for the indeterminacy of the distribution of $X^3$ was provided by the fact that the \textit{Carleman criterion}  (see Section 4.2 of \cite{schmudgen2}), a sufficient condition for \textit{determinacy}, is not satisfied.

The indeterminacy of the distribution of $X^3$ was also established in \cite{stoyanov}, for general $\mathcal{N}(\mu, \sigma^2)$-distributed random variables,  using the \textit{Krein condition} (see  Section 4.3 of \cite{schmudgen2}),  a sufficient condition for \textit{indeterminacy}. It states that if a probability density $f(x)$ on the real line, satisfies
\begin{equation*}
 \int_{-\infty}^{\infty} \frac{\ln f(x)}{1+x^2} \,dx>-\infty
\end{equation*}

\noindent then it is indeterminate.  For the probability density of $X^3$, the explicit calculation (not provided in \cite{{stoyanov}}) is given in the following:

\begin{proposition} Let
\begin{equation*}
 f(x)= \frac{1}{3 \sqrt{\pi}} |x|^{-2/3} \, e^{- |x|^{2/3}} \,,\,x\neq 0
\end{equation*}

\noindent be the probability density of $X^3$ (as in (\ref{dxt})). Then
  \begin{equation*}
 \int_{-\infty}^{\infty} \frac{\ln f(x)}{1+x^2} \,dx= -\pi \left(\ln 3+\frac{1}{2} \ln \pi+2  \right)  >-\infty
\end{equation*}

\noindent so $f$ is indeterminate.

\end{proposition}

\begin{proof} For $x\neq 0$,
\begin{equation*}
\ln\, f(x)= -\ln 3-\frac{1}{2}\ln\pi-\frac{2}{3}\ln |x| - |x|^{2/3}\ .
\end{equation*}

\noindent Thus

\begin{align}\label{wq}
   \int_{-\infty}^{\infty}& \frac{\ln f(x)}{1+x^2} \,dx=2 \,   \int_{0}^{\infty} \frac{\ln f(x)}{1+x^2}
   \,dx\\
   &=2 \left(-(\ln 3+\frac{1}{2}\ln\pi)\int_{0}^{\infty} \frac{1}{1+x^2} \,dx-\frac{2}{3}\int_{0}^{\infty} \frac{\ln x}{1+x^2} \,dx-\int_{0}^{\infty} \frac{x^{2/3}}{1+x^2}
   \,dx\right) \notag \ .
  \end{align}

 \noindent We know that
   \begin{equation*}
 \int_{0}^{\infty} \frac{1}{1+x^2} \,dx=\arctan (\infty)-\arctan (0)=\frac{\pi}{2}
\end{equation*}

\noindent and, letting $x=\frac{1}{t}$, we obtain
 \begin{equation*}
 \int_{0}^{\infty} \frac{\ln x}{1+x^2} \,dx=- \int_{0}^{\infty} \frac{\ln t}{1+t^2} \,dt \implies \int_{0}^{\infty} \frac{\ln x}{1+x^2} \,dx=0 \ .
\end{equation*}

\noindent Letting $x=\tan \theta$, we obtain

\begin{align*}
\int_{0}^{\infty} \frac{x^{2/3}}{1+x^2}  \,dx&=\int_0^{\frac{\pi}{2}} \tan^{2/3}\theta\,d\theta=\int_0^{\frac{\pi}{2}} \sin^{2/3}\theta\, \cos^{-2/3}\theta\,d\theta \\
&=\frac{1}{2} \operatorname{B} (5/6, 1/6)=\frac{1}{2} \Gamma (5/6) \Gamma(1/6)=\frac{1}{2}\frac{\pi}{\sin (\pi/6)}=\pi\notag
\end{align*}

\noindent where, $\operatorname{B}$ and $\Gamma$ are the Beta and Gamma functions, respectively, and we have used the well-known properties
\begin{equation*}
 \operatorname{B}(m,n)=\frac{\Gamma(m) \Gamma(n)}{\Gamma(m+n)}\,,\,\Gamma(1)=1\,,\,\Gamma(z) \Gamma(1-z)=\frac{\pi}{\sin (\pi z)} \ .
\end{equation*}

\noindent Substituting in (\ref{wq}) we obtain
\begin{equation*}
 \int_{-\infty}^{\infty} \frac{\ln f(x)}{1+x^2} \,dx= -\pi \left(\ln 3+\frac{1}{2} \ln \pi+2  \right) \ .
\end{equation*}
\end{proof}

\bibliographystyle{amsplain}

\end{document}